\documentclass[11pt]{amsart}

\usepackage[utf8]{inputenc}
\usepackage[T1]{fontenc}
\usepackage{amsmath, amssymb, amsthm, mathtools}
\usepackage{geometry}
\usepackage[pagebackref]{hyperref}
\usepackage{xcolor}
\usepackage{tikz-cd}
\usepackage{comment}
\usepackage{todonotes}
\usepackage[all]{xy}
\usepackage{enumerate}

\newcommand{\BC}{\mathbb{BC}}       

\newcommand{\C}{\mathbb{C}}          
\newcommand{\R}{\mathbb{R}}
\newcommand{\Z}{\mathbb{Z}}
\newcommand{\D}{\mathbb{D}}
\renewcommand{\P}{\mathbb{P}}
\newcommand{\B}{\operatorname{B}}

\newcommand{\e}{\mathbf{e}}          
\newcommand{\ed}{\mathbf{e}^\dagger} 
\newcommand{\wb}{b}

\newcommand{\iunit}{\mathbf{i}}      
\newcommand{\junit}{\mathbf{j}}      
\newcommand{\kunit}{\mathbf{k}}      

\newcommand{\GL}{\operatorname{GL}}  
\newcommand{\U}{\operatorname{U}}    
\newcommand{\Gr}{\operatorname{Gr}}
\newcommand{\Bun}[1][n]{\operatorname{Bun}_#1^\BC(B)}
\newcommand{\Buninf}{\operatorname{Bun}^\BC(B)}

\newcommand{\ev}{\operatorname{ev}}  
\newcommand{\Tor}{\operatorname{Tor}}
\newcommand{\BCP}{\mathbb{BCP}} 
\newcommand{\CP}{\mathbb{CP}} 

\newtheorem{prop}{Proposition}[section]   
\newtheorem{thm}[prop]{Theorem}           
\newtheorem{lem}[prop]{Lemma}             
\newtheorem{cor}[prop]{Corollary}         
\theoremstyle{definition}
\newtheorem{defn}[prop]{Definition}       
\newtheorem{remark}[prop]{Remark}            
\newtheorem{example}[prop]{Example}       

\title{Bicomplex Characteristic Classes}

\author[Arciniega-Nevárez]{José Antonio Arciniega-Nevárez}
\address{División de Ingenierías, Campus Guanajuato, Universidad de Guanajuato. Av. Juárez No. 77, Zona Centro, C.P. 36000, Guanajuato, Gto., México.}
\email{ja.arciniega@ugto.mx}
\author[Bravo-Ortega]{Yesenia Bravo-Ortega}
\address{Facultad de Ciencias, Universidad de Colima,
Bernal Díaz del Castillo No.~340, Colonia Villas San Sebastián,
C.P.~28045, Colima, Colima, Mexico.}
\email{ybravoo.math@gmail.com}
\author[Rabelo]{Inácio Rabelo}
\address{Departamento de Matemática, Universidade Federal de São Carlos, Caixa Postal 676, 13560-905, São Carlos, SP, Brazil.}
\email{rabeloinacio@alumni.usp.br \\ inaciorabelo@ufscar.br}
\author[Romano-Velázquez]{Agustín Romano-Velázquez}
\address{Instituto de Matem\'aticas, Unidad Cuernavaca, Universidad Nacional Aut\'onoma
de M\'exico, Avenida Universidad s/n, Colonia Lomas de Chamilpa, Cuernavaca,
Morelos, México.}
\email{agustin.romano@im.unam.mx}
\date{September 2026}

\begin{document}

\begin{abstract}
    We introduce bicomplex fiber bundles and develop a framework for bicomplex characteristic classes and almost bicomplex structures on smooth manifolds. In the first part, we present the basic definitions and properties of these bundles, including their fundamental idempotent decomposition and the correspondence between their isomorphism classes and homotopy classes of maps into the bicomplex infinite Grassmannian. We then define bicomplex characteristic classes and show that, in general, they are not determined by the Chern classes of the underlying complex bundle. We also construct the bicomplex Chern character and discuss its algebraic properties. Finally, we apply the resulting theory to almost bicomplex manifolds, providing examples and obstructions to the existence of such structures.
\end{abstract}

\maketitle

\section{Introduction}

The ring of bicomplex numbers $\BC$ is a commutative, unital $\R$-algebra of dimension $4$ that contains nontrivial zero divisors. Its structure is governed by two idempotent elements in $\BC$, denoted by $\e$ and $\ed$, which induce
a canonical splitting
\begin{equation*}
\BC\cong\C\mathbf{e}\oplus\C\ed.
\end{equation*}
This decomposition permeates some algebraic and topological
constructions, allowing bicomplex objects to be studied in terms of
pairs of complex objects. Despite the elegance of this description, a
systematic study of characteristic classes for bicomplex fiber bundles
has remained absent from the literature.

In this article, we develop a comprehensive theory of bicomplex
characteristic classes. We first show that a bicomplex fiber bundle
$E\to B$ of rank $n$ decomposes uniquely, up to isomorphism, as
\begin{equation*}
E\cong E^{1}\mathbf{e}\oplus E^{2}\ed,
\end{equation*}
where $E^{1}$ and $E^{2}$ are complex vector bundles of rank $n$.
We prove a complete classification theorem: isomorphism classes of
rank $m$ bicomplex bundles over a paracompact space $B$ are in
bijection with homotopy classes of maps into the bicomplex
Grassmannian,
\begin{equation*}
[B,\operatorname{Gr}_{m}(\BC^{\infty})]
\longleftrightarrow
\operatorname{Bun}_{m}^{\BC}(B),
\end{equation*}
where
\begin{equation*}
\operatorname{Gr}_{m}(\BC^{\infty})
\cong
\operatorname{Gr}_{m}(\C^{\infty})\times
\operatorname{Gr}_{m}(\C^{\infty}).
\end{equation*}

A key geometric insight is that the bicomplex projective bundle
$\mathbb{P}E\to B$ has fibers isomorphic to
\begin{equation*}
\mathbb{BCP}^{n-1}\cong\mathbb{CP}^{n-1}\times\mathbb{CP}^{n-1}.
\end{equation*}
This is fundamentally different from the complex projective bundle $\mathbb{P}E_\C\to B$ of the underlying complex bundle $E_\C$ (which has fibers $\mathbb{CP}^{2n-1}$) of $E$. This discrepancy is the first
indication that the bicomplex characteristic classes defined here are of interest.

Via a bicomplex version of the Leray--Hirsch theorem applied to
$\mathbb{P}E$, we introduce the bicomplex Chern classes
\begin{equation*}
b_{k}(E)\in H^{2k}(B;\Z),\qquad k=0,\ldots,n,
\end{equation*}
with $b_{0}(E)=2$ and total class
\begin{equation*}
b(E)=b_0+b_{1}(E)+\cdots+b_{n}(E)\in H^{*}(B;\Z).
\end{equation*}
The bicomplex Chern classes are not determined by the Chern classes of the underlying complex bundle $E_{\mathbb C}=E_1\oplus E_2$. The normalization $b_{0}=2$ is forced by the idempotent decomposition:
it is the image of the bicomplex unit $1=\mathbf{e}+\ed$ under the augmentation to ordinary cohomology, reflecting that a
bicomplex bundle of rank $n$ carries two complex components each of rank $n$. We prove that these classes satisfy naturality and normalization,
but not the Whitney sum formula. A splitting principle also reduces their computation to sums of bicomplex line bundles.

We further construct the bicomplex Chern character
\begin{equation*}
\operatorname{ch}_{\BC}:
K_{\BC}(B)\longrightarrow H^{\mathrm{ev}}(B;\mathbb{Q}),
\end{equation*}
showing that it is a group homomorphism. We carefully analyze the
structure of the bicomplex $K$-theory of connected base space $B$, noting that
\begin{equation*}
K_{\BC}(B)\cong K(B)\times_{\Z}K(B),
\end{equation*}
the fiber product over the rank map, not the tensor product $K(B)\otimes K(B)$. Consequently, $\operatorname{ch}_{\BC}$ is not a rational isomorphism.

Finally, we apply our theory to the study of almost bicomplex manifolds. In particular, we
prove that the complex projective plane $\mathbb{CP}^{2}$ admits no almost bicomplex structure, a result which, to our knowledge, is new in the literature. 

This paper is organized as follows. Section~2 recalls the algebraic
structure of bicomplex modules and the basic linear algebra.
Section~3 introduces bicomplex fiber bundles, their operations, and
their homotopy classification. In Section~4, we develop the theory of
bicomplex Chern classes, proving their fundamental properties, the
splitting principle, and their independence from complex Chern classes of $E_\C$.
Section~5 is devoted to the universal Chern classes and the cohomology
of the classifying space. In Section~6, we construct the bicomplex
Chern character and discuss its algebraic properties. Finally,
Section~7 applies our results to almost bicomplex manifolds and
provides explicit examples of the new obstructions.

\section{Bicomplex modules}\label{sec:bicomplex-modules}
In this section, we shall give an overview of the algebraic and linear structures that form the foundation for the rest of the paper.
\subsection{The ring of bicomplex numbers}

The ring of bicomplex numbers is defined as
\begin{equation*}
    \BC = \{ Z = z + w\,\junit \mid z,w \in \C(\iunit),\ \junit^2 = -1 \},
\end{equation*}
where $\C(\iunit) = \{ a + b\,\iunit \mid a,b \in \R,\ \iunit^2 = -1 \}$ denotes the usual complex numbers. Addition and multiplication are performed componentwise, with $\kunit = \iunit\,\junit = \junit\,\iunit$. Since both imaginary units commute, one readily checks that $\kunit^2 = 1$, so $\BC$ contains nontrivial zero divisors. The ring $\BC$ can be regarded as a real algebra of dimension $4$, or as a complex algebra of dimension $2$ over $\C(\iunit)$. A fundamental feature of $\BC$ is the existence of two distinguished idempotent elements.

\begin{prop}
    The elements
    \begin{equation*}
        \e = \frac{1 + \iunit\,\junit}{2} \quad \text{and} \quad
        \ed = \frac{1 - \iunit\,\junit}{2}
    \end{equation*}
    satisfy the following properties:
    \begin{enumerate}
        \item $\e\,\ed = \ed\,\e = 0$ (they are orthogonal zero divisors);
        \item $\e^2 = \e$ and $(\ed)^2 = \ed$ (idempotency);
        \item $\e + \ed = 1$ and $\e - \ed = \iunit\,\junit$.
    \end{enumerate}
\end{prop}

A direct consequence is the \emph{idempotent decomposition} of any bicomplex number. If $Z = z + w\,\junit$, then
\begin{equation*}
    Z = Z\e + Z\ed = (z - w\,\iunit)\e + (z + w\,\iunit)\ed.
\end{equation*}
Geometrically, this exhibits $\BC$ as the direct sum of two copies of $\C$:
\begin{equation*}
    \BC \cong \C\,\e \oplus \C\,\ed.
\end{equation*}
This decomposition is one of the main tools in bicomplex analysis, and the theory discussed in the sequel is based on it. 

\subsection{The canonical module \texorpdfstring{$\BC^n$}{}}

Let $n \geq 1$. The canonical bicomplex module $\BC^n$ is defined as the set of $n$-tuples with entries in $\BC$:
\begin{equation*}
    \BC^n = \{ v = (Z_1, \dots, Z_n) \mid Z_k \in \BC \}.
\end{equation*}
Applying the idempotent decomposition componentwise, any vector $v \in \BC^n$ splits as
\begin{equation*}
    v = v^1\e + v^2\ed,
\end{equation*}
with $v^1, v^2 \in \C^n$. This gives the fundamental canonical isomorphism
\begin{equation*}
    \BC^n \cong \C^n\e \oplus \C^n\ed.
\end{equation*}
This decomposition is compatible with all linear and topological structures; in particular, $\BC^n$ inherits the product topology of $\R^{4n}$.

\subsection{Null cone and regular elements}\label{subsec:null-cone}

The zero divisors of $\BC$ are precisely the elements for which one of the complex components vanishes. This motivates the following definition.

\begin{defn}
    The \emph{null cone} of $\BC^n$ is the set
    \begin{equation*}
        \mathcal{N}_n = \{ v = v^1\e + v^2\ed \in \BC^n \mid v^1 = 0 \text{ or } v^2 = 0 \text{ in } \C^n \}.
    \end{equation*}
    A vector $v \notin \mathcal{N}_n$, i.e. with both $v^1$ and $v^2$ non-zero, is called a \emph{regular vector}.
\end{defn}

For $n=1$, the null cone is precisely the set of zero divisors $\C\e \cup \C\ed$. The group of units of $\BC$, denoted $\BC^*$, is therefore
\begin{equation*}
    \BC^* \cong \C^*\e \oplus \C^*\ed.
\end{equation*}
Indeed, the inverse of a unit $Z = z + w\,\junit$ is given by
\begin{equation*}
    Z^{-1} = \frac{\widetilde{Z}}{z^2 + w^2}
            = \frac{1}{z - w\,\iunit}\,\e + \frac{1}{z + w\,\iunit}\,\ed,
\end{equation*}
where $\widetilde{Z} = z - w\,\junit$ denotes the conjugate of $Z$ with respect to $\junit$.

\subsection{Inner product and norm}

The standard bicomplex inner product is defined as follows.

\begin{defn}\label{def:inner-product}
    For $u = u^1\e + u^2\ed$ and $v = v^1\e + v^2\ed$ in $\BC^n$, define
    \begin{equation*}
        \langle u, v \rangle_{\BC}
        = \langle u^1, v^1 \rangle_{\C} \e
          + \langle u^2, v^2 \rangle_{\C} \ed,
    \end{equation*}
    where $\langle \cdot, \cdot \rangle_{\C}$ denotes the standard complex Hermitian inner product.
\end{defn}

This inner product is positive definite in the sense that $\langle v, v \rangle_{\BC} = 0$ if and only if $v = 0$. Moreover, the \emph{induced bicomplex norm}, taking values in the hyperbolic numbers $\D = \{ a + b\,\kunit \mid a,b \in \R \}$, is given by
\begin{equation*}
    \| v \|_{\D} = \| v^1 \|_{\C}\,\e + \| v^2 \|_{\C}\,\ed. 
\end{equation*}
See \cite{Banerjee2019} for more details.
\subsection{Gram--Schmidt process}\label{subsec:gram-schmidt}

The idempotent decomposition makes the classical Gram--Schmidt process completely reducible. For regular vectors $u, v \in \BC^n$, the orthogonal projection of $u$ onto $v$ decomposes componentwise:
\begin{equation*}
    \frac{\langle u, v \rangle_{\BC}}{\|v\|^2_{\D}} \, v
    =
    \frac{\langle u^1, v^1 \rangle_{\C}}{\|v^1\|^2_{\C}} \, v^1\,\e
    +
    \frac{\langle u^2, v^2 \rangle_{\C}}{\|v^2\|^2_{\C}} \, v^2\,\ed.
\end{equation*}
Consequently, one can apply the complex Gram--Schmidt process independently to the $\e$-part and the $\ed$-part. This yields a bicomplex orthonormal basis for any bicomplex submodule spanned by regular vectors.

\subsection{Matrix groups and unitary group}

The idempotent decomposition extends naturally to matrices. For $M, N \in \mathrm{M}_n(\BC)$, write $M = M^1\e + M^2\ed$ and $N = N^1\e + N^2\ed$ with $M^1$, $M^2$, $N^1$, $N^2$ in $\mathrm{M}_n(\C)$. Their product satisfies
\begin{equation}\label{eq:matrix-multiplication}
    MN = (M^1 N^1)\e + (M^2 N^2)\ed.
\end{equation}

\begin{defn}
    The \emph{bicomplex general linear group} is
    \begin{equation*}
        \mathrm{GL}(n,\BC)
        = \{ A \in \mathrm{M}_n(\BC) \mid \det(A) \text{ is a unit} \}.
    \end{equation*}
\end{defn}

By Equation~\eqref{eq:matrix-multiplication}, a matrix is invertible if and only if both complex components $A^1$ and $A^2$ are invertible. Hence
\begin{equation*}
    \mathrm{GL}(n,\BC) \cong \mathrm{GL}(n,\C)\,\e \oplus \mathrm{GL}(n,\C)\,\ed.
\end{equation*}

Similarly, if $M = (Z_{ij})$, let $\overline{M} = (\overline{Z_{ij}})$ be the componentwise complex conjugation, and define $M^*$ as the conjugate transpose. The unitary group is
\begin{equation*}
    \mathrm{U}(n,\BC) = \{ U \in \mathrm{M}_n(\BC) \mid U U^* = I_n \}.
\end{equation*}
Again, from the idempotent decomposition we have
\begin{equation}
    \mathrm{U}(n,\BC) \cong \mathrm{U}(n,\C)\,\e \oplus \mathrm{U}(n,\C)\,\ed.
\end{equation}

Using the inner product in Definition~\ref{def:inner-product} and the matrix multiplication formula of Equation~\eqref{eq:matrix-multiplication}, it is straightforward to verify that unitary bicomplex matrices preserve the bicomplex inner product:
\begin{equation*}
    \langle U u, U v \rangle_{\BC} = \langle u, v \rangle_{\BC}, \qquad U \in \mathrm{U}(n,\BC),\ u,v \in \BC^n.
\end{equation*}

The bicomplex conjugation $Z\mapsto\widetilde{Z}=z-w\mathbf{j}$ introduced in \S\textup{\ref{subsec:null-cone}} interchanges the idempotents: $\widetilde{\e}=\ed$ and $\widetilde{\ed}=\e$.
Applied entrywise to matrices, it induces an involution that we use below to prove the invariance of the characteristic classes under the following  map
\begin{equation}\label{eq:involutionmat}
\begin{gathered}
\iota: \GL(n,\BC)\to\GL(n,\BC),\\
A^1\e+A^2\ed \mapsto A^2\e+A^1\ed,
\end{gathered}   
\end{equation}
which swaps the two factors in the idempotent decomposition of $\GL(n,\BC)$.

\section{Bicomplex fiber bundles}\label{sec:bicomplex-fiber-bundles}

In this section, we introduce the notion of a bicomplex fiber bundle, establish its fundamental idempotent decomposition, and derive the basic properties that will be used throughout the paper. Our treatment closely follows the classical theory of complex vector bundles, with the crucial difference that the structure group is the bicomplex general linear group $\GL(n,\BC)$. The main reference for the bundle-theoretic arguments is \cite[Chapter~1]{Ati1967}.

\subsection{Definition and first examples}

Let $B$ be a paracompact Hausdorff topological space.

\begin{defn}
    A \emph{bicomplex fiber bundle of rank $n$} over $B$ is a fiber bundle $\pi:E \to B$ with typical fiber $\BC^n$ and structure group $\GL(n,\BC)$. Equivalently, there exists an open cover $\{U_\alpha\}$ of $B$ and trivializations
    \begin{equation*}
        \phi_\alpha : \pi^{-1}(U_\alpha) \to U_\alpha \times \BC^n
    \end{equation*}
    such that the transition functions
    \begin{equation*}
        g_{\alpha\beta} : U_\alpha \cap U_\beta \to \GL(n,\BC)
    \end{equation*}
    are continuous and satisfy the usual cocycle conditions, and they also preserve the bicomplex structure.
\end{defn}

\begin{example}
    The \emph{trivial bundle} $E = B \times \BC^n$ with the projection onto the first factor is a bicomplex fiber bundle of rank $n$.
\end{example}

\begin{example}\label{ex: universal}
Let $\underline{\mathbb C}\e = B \times \mathbb C\e$ and $\underline{\mathbb C}\ed = B \times \mathbb C\ed$ denote the trivial complex line bundles whose fibers are the one-dimensional complex vector spaces $\mathbb C\e$ and $\mathbb C\ed$, respectively. Given a complex vector bundle $E^1 \to B$, we define
\begin{equation*}
    E^1 \e = E^1 \otimes_{\mathbb C} \underline{\mathbb C}\e.
\end{equation*}
Similarly, for a complex bundle \(E^2 \to B\),
\begin{equation*}
    E^2 \ed = E^2 \otimes_{\mathbb C} \underline{\mathbb C}\ed.
\end{equation*}
Since $\underline{\mathbb C}\e$ and $\underline{\mathbb C}\ed$ are trivial as complex line bundles, we have isomorphisms $E^1\e \cong E^1$ and $E^2\ed \cong E^2$ as complex vector bundles. Nevertheless, the notation $E^1\e$ and $E^2\ed$ is useful because it explicitly records the action of the idempotents: an element $v^1\e \in E^1\e$ is multiplied by \(\e\) as $(v^1\e)\e = v^1\e^2 = v^1\e$, and similarly for \(\ed\). We shall use this shorthand throughout the paper.

Let $\pi^1:E^1 \to B$ and $\pi^2:E^2 \to B$ be complex vector bundles of rank $n$. Then
\begin{equation*}
    E = (E^1 \otimes_{\mathbb C} \underline{\mathbb C}\e) \oplus (E^2 \otimes_{\mathbb C} \underline{\mathbb C}\ed)
\end{equation*}
is a bicomplex fiber bundle of rank $n$, with the bicomplex structure defined by the idempotent action. In the previous notation, we write
\begin{equation*}
    E \cong E^1\e \oplus E^2\ed.
\end{equation*}    
\end{example}

\begin{defn}\label{def:isomorphism}
    A homomorphism $\phi : E \to F$ between bicomplex fiber bundles $\pi : E \to B$ and $\pi' : F \to B$ is called a bicomplex homomorphism if for each $x \in B$, the map $\phi_{x} : E_{x} \to F_{x}$ is bicomplex linear. If, in addition, $\phi$ is an isomorphism and $\phi_{x}$ is a bicomplex isomorphism for each $x \in B$, then it is called a bicomplex isomorphism. 
\end{defn}

Henceforth, by isomorphism we always mean a bicomplex isomorphism, unless otherwise stated. We also denote by $\Buninf$ the set of isomorphism classes of bicomplex fiber bundles on $B$ and $\Bun$ the subset of $\Buninf$ given by bicomplex fiber bundles of rank $n$.

\subsection{Idempotent decomposition of bicomplex bundles}

Now, we write the converse of Example~\ref{ex: universal}, then we have,
\begin{prop}\label{prop:idempotent_bundles}
    Let \(\pi:E \to B\) be a bicomplex fiber bundle of rank \(n\). Then there exist complex vector bundles \(\pi^1:E^1 \to B\) and \(\pi^2:E^2 \to B\), each of rank \(n\), and an isomorphism of bicomplex bundles
    \begin{equation*}
        E \cong E^1\e \oplus E^2\ed.
    \end{equation*}
    Moreover, this decomposition is canonical up to isomorphism.
\end{prop}

\begin{proof}
  We define the complex bundles
  \begin{equation*}
  E^1\e= E \otimes_{\BC} \underline{\C}\e \quad \text{and}\quad
  E^2\ed= E \otimes_{\BC} \underline{\C}\ed.
  \end{equation*}
  Since $\BC \cong \C\mathbf{e} \oplus \C\ed$ as $\BC$-modules and tensor products distribute over direct sums, the canonical isomorphism $E \cong E \otimes_{\BC} \BC$ induces
  \begin{equation*}
  E \cong (E \otimes_{\BC} \C\mathbf{e}) \oplus (E \otimes_{\BC} \C\ed) = E^1\e \oplus E^2\ed.
  \end{equation*}
  This decomposition is intrinsic, as it depends only on the algebraic splitting of the ring $\BC$ and not on any choice of cover or trivializations.

  To verify that $E^1\mathbf{e}$ and $E^2\ed$ are complex bundles of rank $n$, let $\{U_\alpha\}$ be a trivializing cover for $E$ with transition functions $g_{\alpha\beta}: U_{\alpha\beta} \to \mathrm{GL}(n,\BC)$. The idempotent decomposition writes $g_{\alpha\beta}=g_{\alpha\beta}^1\e+g_{\alpha\beta}^2\ed$, with $g_{\alpha\beta}^1,g_{\alpha\beta}^2\in \mathrm{GL}(n,\C)$. Under the local trivialization $\phi_\alpha: E|_{U_\alpha} \to U_\alpha \times \BC^n$, the induced bundle $(E^1\mathbf{e})|_{U_\alpha}$ is identified with $U_\alpha \times \C^n$ via
  \begin{gather*}
  \phi_\alpha^1: (E^1\e)|_{U_\alpha} \to U_\alpha \times \C^n, \\ v^1\e \mapsto v^1.
  \end{gather*}
  A direct computation shows that the transition function of $E^1\e$ on $U_{\alpha\beta}$ is precisely $g_{\alpha\beta}^1$; similarly, that of $E^2\ed$ is $g_{\alpha\beta}^2$. Hence both are complex bundles of rank $n$.

  Finally, since the construction via tensor products with the global bundles $\underline{\C}\e$ and $\underline{\C}\ed$ involves no coordinate choices, the decomposition is canonical up to isomorphism. 
\end{proof}

We shall refer to $E^1$ and $E^2$ as the \emph{idempotent complex bundles} associated to $E$.

If $\phi:E \to F$ is a homomorphism of bicomplex fiber bundles, then its restriction to each fiber $\phi_x:E_x \to F_x$ is $\BC$-linear and hence decomposes as $\phi_x = \phi_x^1 \e + \phi_x^2 \ed$. Consequently, the homomorphism itself admits an idempotent decomposition
    \begin{equation*}
        \phi = \phi^1 \e + \phi^2 \ed,
    \end{equation*}
    where $\phi^1:E^1 \to F^1$ and $\phi^2:E^2 \to F^2$ are complex linear maps.
 
\begin{remark}\label{remark:idptbundles} According to Definition \ref{def:isomorphism}, two bicomplex fiber bundles over a space $B$ are bicomplex isomorphic if and only if the associated idempotent complex bundles are complex isomorphic. In particular, a bicomplex fiber bundle is trivial if and only if its idempotent complex bundles are trivial.
\end{remark}

\subsection{Bundle operations}

The standard operations on fiber bundles are compatible with the idempotent decomposition and will be used extensively in the sequel. Let $E = E^1 \e \oplus E^2 \ed$ and $F = F^1 \e \oplus F^2 \ed$ be bicomplex bundles over $B$.

\begin{itemize}
    \item \textbf{Direct sum:}
    \begin{equation*}
        E \oplus F = (E^1 \oplus F^1)\e \oplus (E^2 \oplus F^2)\ed.
    \end{equation*}

    \item \textbf{Tensor product:} Using the idempotent decomposition of $\BC$-modules, mixed tensor products vanish, yielding
    \begin{equation*}
        E \otimes_{\BC} F = (E^1 \otimes_\C F^1)\e \oplus (E^2 \otimes_\C F^2)\ed. 
    \end{equation*}

    \item \textbf{Pullback:} For a continuous map $f:B' \to B$,
    \begin{equation*}
        f^*(E) = f^*(E^1)\e \oplus f^*(E^2)\ed.
    \end{equation*}
    \item \textbf{Dual:} If $E^*$ denotes the dual bicomplex fiber bundle,
    \begin{equation*}
        E^*\cong {E^1}^*\e\oplus {E^2}^*\ed.
    \end{equation*}
\end{itemize}

\subsection{Sections and triviality criterion}

Let $\Gamma(B,E)$ denote the set of bicomplex sections of $E \to B$. Using the projections $p_1:E \to E^1\e$ and $p_2:E \to E^2\ed$, any section $s:B \to E$ decomposes as
\begin{equation*}
    s = s^1 \e + s^2 \ed,
\end{equation*}
where $s^1 \in \Gamma(B,E^1)$ and $s^2 \in \Gamma(B,E^2)$ are complex sections.

\begin{example}\label{ex:regular_from_zero_divisors}
Let $B=\mathbb{C}\mathbb{P}^{1}\times\mathbb{C}\mathbb{P}^{1}$ and denote by 
$p_{1},p_{2}\colon B\to\mathbb{C}\mathbb{P}^{1}$ the projections onto the two factors. 
Let $\gamma_{\mathbb{C}}\to\mathbb{C}\mathbb{P}^{1}$ be the complex tautological line bundle.

Consider the two bicomplex line bundles over $B$ defined by
\begin{equation*}
L_{1}=\bigl(p_{1}^{*}\gamma_{\C}\bigr)\e\oplus\underline{\C}\ed
\quad\text{and}\quad
L_{2}=\underline{\mathbb{C}}\,\mathbf{e}\;\oplus\;\bigl(p_{2}^{*}\gamma_{\mathbb{C}}\bigr)\mathbf{e}^{\dagger},
\end{equation*}
and set $E=L_{1}\oplus L_{2}$, a bicomplex bundle of rank $2$.

Since $\gamma_{\mathbb{C}}$ has no non-zero global sections, every section of $L_{1}$ is of the form 
$\sigma^{1}\mathbf{e}^{\dagger}$ with $\sigma^{1}\in\Gamma(B,\underline{\mathbb{C}})$; likewise, every section of $L_{2}$ is of the form 
$\sigma^{2}\mathbf{e}$ with $\sigma^{2}\in\Gamma(B,\underline{\mathbb{C}})$. 
Take the constant sections
\begin{equation*}
s^{1}=0\cdot\e+1\cdot\ed\in\Gamma(B,L_{1})\quad\text{and}\quad
s^{2}=1\cdot\e+0\cdot\ed\in\Gamma(B,L_{2}).
\end{equation*}
Viewed as sections of the summands of $E$, both $s^{1}$ and $s^{2}$ are pointwise zero divisors: 
at each $x\in B$ we have $s^{1}(x)\in\mathbb{C}\mathbf{e}^{\dagger}$ and $s^{2}(x)\in\mathbb{C}\mathbf{e}$.

Now form the global section $s=(s^{1},s^{2})\in\Gamma(B,E)$. 
Using the canonical identification $E_{x}\cong\mathbb{B}\mathbb{C}^{2}$, its idempotent decomposition at any point is
\begin{equation*}
s(x)=\begin{pmatrix}0\\ 1\end{pmatrix}\mathbf{e}+\begin{pmatrix}1\\ 0\end{pmatrix}\mathbf{e}^{\dagger}.
\end{equation*}
Both complex components are non-zero vectors in $\mathbb{C}^{2}$, hence $s(x)\notin\mathcal{N}_{2}$ for every $x\in B$. Therefore $s$ is a regular section of $E$.

This shows that regularity is a property of the idempotent decomposition of the total vector in the fiber, not of the individual summands: two pointwise zero-divisor sections can combine to produce a regular section     of the direct sum.
\end{example}

\begin{defn}
    A section $s \in \Gamma(B,E)$ is called \emph{regular} if $s(x)$ is a regular vector in the fiber $E_x \cong \BC^n$ for every $x \in B$. Equivalently, both complex sections $s^1$ and $s^2$ are nowhere vanishing.
\end{defn}

The following criterion generalizes the classical characterization of trivial complex vector bundles.

\begin{prop}
    A bicomplex fiber bundle $\pi:E \to B$ of rank $n$ is trivial if and only if it admits $n$ linearly independent regular sections $s_1,\dots,s_n$.
\end{prop}

\begin{proof}
    Define a map
    \begin{gather*}
        \phi : B \times \BC^n \to E, \\ 
        \phi(x,(v_1,\dots,v_n)) = \sum_{i=1}^n v_i s_i(x).
    \end{gather*}
    Since each $s_i$ is regular and linearly independent, the fibers $\phi_x : \BC^n \to E_x$ are isomorphisms of $\BC$-modules for every $x \in B$. 
    
    To see that \(\phi\) is a fiberwise isomorphism, let \(A(x) \in \mathrm{M}_n(\BC)\) be the matrix whose columns are the coordinates of \(s_1(x),\dots,s_n(x)\) in a bicomplex trivialization of \(E_x\). The idempotent decomposition gives
\begin{equation*}
A(x) = A^1(x)\e + A^2(x)\ed,
\end{equation*}
Since the sections $\{s_i\}$ are linearly independent, their idempotent
components $\{s_i^1\}$ and $\{s_i^2\}$ are linearly independent in
$\mathbb{C}^n$. Hence the matrices $A^1$ and $A^2$ are invertible.
Therefore, the bicomplex matrix $A$ is invertible.
Therefore, \(\phi_x: \BC^n \to E_x\) is an isomorphism of \(\BC\)-modules for every \(x \in B\).
    
\end{proof}

\subsection{Bicomplex bundles over compact spaces}

We now collect several useful facts for bundles over compact base spaces. These are direct consequences of the idempotent decomposition and the corresponding results for complex vector bundles.

\begin{lem}
    Let $B$ be compact, $B' \subset B$ a closed subspace, and $\pi:E \to B$ a bicomplex fiber bundle.
    \begin{enumerate}
        \item Any bicomplex section $s:B' \to E|_{B'}$ can be continuously extended to a section over all of $B$.
        \item Let $\pi':F \to B$ be another bicomplex fiber bundle. If $\phi:E|_{B'} \to F|_{B'}$ is an isomorphism, then there exists an open neighborhood $U$ of $B'$ in $B$ and an isomorphism $\widetilde{\phi}:E|_U \to F|_U$ extending $\phi$.
    \end{enumerate}
\end{lem}

\begin{cor}\label{cor:homotopy_compact}
    Let $B$ and $B'$ be compact spaces.
    \begin{enumerate}
        \item If $f_t:B' \to B$ is a homotopy, then $f_0^*(E) \cong f_1^*(E)$.
        \item A homotopy equivalence $f:B' \to B$ induces a bijection $f^*:\Bun[n]\to\operatorname{Bun}_n^\BC(B')$.
        \item If $B$ is contractible, then every bicomplex fiber bundle over $B$ is trivial.
    \end{enumerate}
\end{cor}

The proof of the lemma follows by applying the complex version (\cite[Lemmas 1.4.1 and 1.4.2]{Ati1967}) to the idempotent components $E^1$ and $E^2$, and then taking the product of the resulting extensions. The corollary is then standard.

\subsection{Ample subspaces and finite-dimensional reductions}

Following \cite[Lemma~1.4.12]{Ati1967}, we recall the notion of an ample subspace of sections.

\begin{defn}
    A subspace $V \subset \Gamma(B,E)$ is called \emph{ample} if the evaluation map
    \begin{gather*}
        \ev : B \times V \to E, \\ (x,s) \mapsto s(x)
    \end{gather*}
    is surjective.
\end{defn}

Using a partition of unity subordinate to a trivializing cover, one can construct ample subspaces locally and then glue them together.

\begin{prop}
    Let $\pi:E \to B$ be a bicomplex fiber bundle. Then there exists an ample subspace $V \subset \Gamma(B,E)$. If $B$ is compact, one may choose $V$ to be finite-dimensional.
\end{prop}

\begin{cor}\label{cor:epimorphism}
    Let $\pi:E \to B$ be a bicomplex bundle, where $B$ is compact.
    \begin{enumerate}
        \item There exists an epimorphism $\phi:B \times \BC^n \to E$ for some integer $n$.
        \item There exists a bicomplex bundle $F \to B$ such that $E \oplus F$ is trivial.
    \end{enumerate}
\end{cor}

The first part follows from the existence of a finite-dimensional ample $V$ by taking $n = \dim V$ and identifying $B \times V$ with $B \times \BC^n$. The second part is obtained by taking $F$ to be the kernel of the epimorphism, which splits since $\BC^n$ is a projective $\BC$-module.

\section{Bicomplex Grassmannian and homotopy theorem}

In this section, we provide a homotopical classification of bicomplex fiber bundles. The main result is that isomorphism classes of bicomplex fiber bundles of rank $n$ over a paracompact space $B$ are in bijection with homotopy classes of maps from $B$ into the infinite bicomplex Grassmannian $\Gr_m(\BC^\infty)$. This is the bicomplex analogue of the classical classification theorem for complex vector bundles.

\subsection{The bicomplex Grassmannian}

Let $0\leq m \le n$ be positive integers.

\begin{defn}
    The \emph{bicomplex Grassmannian} $\Gr_m(\BC^n)$ is the set of all free $\BC$-submodules of rank $m$ of $\BC^n$.
\end{defn}

Similarly, let $\mathrm{V}_m(\BC^n)$ denote the \emph{Stiefel manifold} of $m$-frames in $\BC^n$; that is, the space of $m$-tuples of linearly independent regular vectors in $\BC^n$. There is a natural surjection
\begin{equation*}
    \mathrm{V}_m(\BC^n) \to \Gr_m(\BC^n)
\end{equation*}
that sends an $m$-frame to the free $\BC$-submodule it spans. We give $\Gr_m(\BC^n)$ the quotient topology induced by this map.

The following proposition shows that the bicomplex Grassmannian is, topologically, a product of two complex Grassmannians.

\begin{prop}
    There is a homeomorphism
    \begin{equation*}
        \Gr_m(\BC^n) \cong \Gr_m(\C^n) \times \Gr_m(\C^n).
    \end{equation*}
\end{prop}

\begin{proof}
    Let $V \subset \BC^n$ be a free $\BC$-submodule of rank $m$. By the idempotent decomposition, $V$ decomposes as
    \begin{equation*}
        V = V^1\e \oplus V^2\ed,
    \end{equation*}
    where $V^1 \subset \C^n$ and $V^2 \subset \C^n$ are complex subspaces. Since $V$ has rank $m$ over $\BC$, and $\BC \cong \C\e \oplus \C\ed$, we have $\dim_\C V^1 = \dim_\C V^2 = m$. Conversely, given two complex subspaces $W^1, W^2 \subset \C^n$ of dimension $m$, the subspace $W^1\e \oplus W^2\ed \subset \BC^n$ is a free $\BC$-submodule of rank $m$. This bijection is continuous in both directions, hence gives the stated homeomorphism.
\end{proof}

Taking the direct limit as $n \to \infty$ yields the infinite bicomplex Grassmannian
\begin{equation*}
    \Gr_m(\BC^\infty) \cong \Gr_m(\C^\infty) \times \Gr_m(\C^\infty).
\end{equation*}

\subsection{The tautological bicomplex fiber bundle}

Over the finite Grassmannian $\Gr_m(\BC^n)$, we have a canonical bundle.

\begin{defn}
    The \emph{tautological bicomplex fiber bundle} $\gamma_{\BC}^{m,n}$ is the subbundle of the bicomplex trivial bundle $\Gr_m(\BC^n) \times \BC^n$ defined by
    \begin{equation*}
        \gamma_{\BC}^{m,n} = \bigl\{ (V, v) \in \Gr_m(\BC^n) \times \BC^n \mid v \in V \bigr\}.
    \end{equation*}
    Its projection is $\pi_{\BC}^{m,n} : \gamma_{\BC}^{m,n} \to \Gr_m(\BC^n)$, $(V,v) \mapsto V$.
\end{defn}

Using the idempotent decomposition of the Grassmannian, the bicomplex tautological fiber bundle decomposes as
\begin{equation*}
    \gamma_{\BC}^{m,n} = \gamma^{m,n}_\C\e \oplus \gamma^{m,n}_\C\ed,
\end{equation*}
where $\gamma^{m,n}_\C$ denotes the usual complex tautological vector bundle over $\Gr_m(\C^n)$, and the projection is
\begin{equation*}
    \pi_{\BC}^{m,n} = (\pi_{\C}^{m,n}, \pi_{\C}^{m,n}).
\end{equation*}

Passing to the limit $n \to \infty$, we obtain the infinite tautological bundle $\gamma_{\BC}^m \to \Gr_m(\BC^\infty)$.

\subsection{The classification theorem}

We are now ready to state the main theorem of this section. Recall that we have denoted by $\Bun$ the set of isomorphism classes of rank-$m$ bicomplex fiber bundles over $B$.

\begin{thm}\label{thm:classification}
    Let $B$ be a paracompact Hausdorff space. There is a bijection
    \begin{equation*}
        [B, \Gr_m(\BC^\infty)] \to \Bun,
    \end{equation*}
    given by
    \begin{equation}
        [f] \mapsto f^*(\gamma_{\BC}^m).
    \end{equation}
\end{thm}

Before giving the proof, we recall the necessary ingredients.

First, by Corollary~\ref{cor:homotopy_compact}, homotopic maps $f_0, f_1: B \to \Gr_m(\BC^\infty)$ induce isomorphic pullback bundles. Hence the map $[f] \mapsto f^*(\gamma_{\BC}^m)$ is well-defined.

Second, by Corollary~\ref{cor:epimorphism}, if $B$ is compact, every bicomplex bundle $E \to B$ admits an epimorphism
\begin{equation*}
    \phi: B \times \BC^N \to E
\end{equation*}
for some integer $N$. In the general paracompact case, the same result holds by using a partition of unity argument (see \cite[Proposition~1.4.15]{Ati1967}).

Given such an epimorphism, for each $x \in B$, the kernel $\ker(\phi_x)$ is a $\BC$-submodule of $\BC^N$ of rank $N-m$. Thus we obtain a map
\begin{align*}
    \iota_\phi: B &\to \Gr_{N-m}(\BC^N) \subset \Gr_{N-m}(\BC^\infty),
    \\ x &\mapsto \ker(\phi_x).
\end{align*}

Let $\phi:B\times\mathbb{B}\mathbb{C}^{N}\to E$ be an epimorphism of bicomplex bundles.
For each $x\in B$, let $\ker(\phi_{x})^{\perp}\subset\mathbb{B}\mathbb{C}^{N}$ denote the
orthogonal complement of\/ $\ker(\phi_{x})$ with respect to the standard bicomplex
inner product of Definition~\ref{def:inner-product}.
Then
\begin{lem}\label{lem:classifying_map}
The map
\begin{gather*}
\iota_{\phi}:B\to\Gr_{m}(\BC^{N}),\\
x\mapsto\ker(\phi_{x})^{\perp},    
\end{gather*}
is well defined and continuous.  Moreover, there is a natural isomorphism
\[
\iota_{\phi}^{*}(\gamma_{\mathbb{B}\mathbb{C}}^{m,N})\cong E.
\]
\end{lem}

\begin{proof}
By the bicomplex Gram--Schmidt process (Subsection~\ref{subsec:gram-schmidt}),
$\BC^{N}=\ker(\phi_{x})\oplus\ker(\phi_{x})^{\perp}$ and the bicomplex rank of $\ker(\phi_{x})^{\perp}$ is $m$ for every $x\in B$. 
By definition,
\begin{equation*}
\iota_{\phi}^{*}(\gamma_{\mathbb{B}\mathbb{C}}^{m,N})
=\bigl\{(x,v)\in B\times\mathbb{B}\mathbb{C}^{N}\mid v\in\ker(\phi_{x})^{\perp}\bigr\}.
\end{equation*}
Since $\phi_{x}$ is surjective and $\ker(\phi_{x})\cap\ker(\phi_{x})^{\perp}=\{0\}$,
the restriction $\phi_{x}|_{\ker(\phi_{x})^{\perp}}:\ker(\phi_{x})^{\perp}\to E_{x}$
is an isomorphism of $\mathbb{B}\mathbb{C}$-modules for each $x$.  These fiberwise
isomorphisms assemble into a global bundle isomorphism
$\iota_{\phi}^{*}(\gamma_{\mathbb{B}\mathbb{C}}^{m,N})\cong E$.
\end{proof}

\begin{proof}[Proof of Theorem~\ref{thm:classification}]
    We have already shown that the map $[f] \mapsto f^*(\gamma_{\BC}^m)$ is well-defined.

To prove surjectivity, let $E\to B$ be a bicomplex fiber bundle of rank $m$. Choose an
epimorphism $\phi:B\times\BC^N\to E$ as above, and consider the induced map
\begin{align*}
\iota_{\phi}:B&\to\Gr_m(\BC^{N})
\subset\Gr_m(\BC^{\infty}),\\
x&\mapsto\ker(\phi_{x})^{\perp}.
\end{align*}
By Lemma~\ref{lem:classifying_map}, $\iota_{\phi}^{*}(\gamma_{\mathbb{B}\mathbb{C}}^{m})\cong E$.

To prove injectivity, suppose $f, g: B \to \Gr_m(\BC^\infty)$ satisfy $f^*(\gamma_{\BC}^m) \cong g^*(\gamma_{\BC}^m)$. By compactness, we may assume their images lie in a finite Grassmannian $\Gr_m(\BC^N)$. The isomorphism of the pullback bundles gives an isomorphism of the associated epimorphisms from the trivial bundle to the pullback. Standard arguments (see \cite[Theorem~1.4.15]{Ati1967}) show that the maps $f$ and $g$ are homotopic. Alternatively, one can use the fact that the space of epimorphisms $B \times \BC^N \to E$ is path-connected, since $\GL(N,\BC)$ is connected (indeed, $\GL(N,\BC) \cong \GL(N,\C)\e \oplus \GL(N,\C)\ed$, and $\GL(N,\C)$ is connected). Hence the induced maps to the Grassmannian are homotopic.
\end{proof}

\begin{remark}[Idempotent decomposition of the classifying map]
If $E=E^{1}\e\oplus E^{2}\ed$, then any epimorphism
$\phi:B\times\BC^{N}\to E$ decomposes as
$\phi=\phi^{1}\e+\phi^{2}\ed$, where
$\phi^{1}:B\times\C^{N}\to E^{1}$ and
$\phi^{2}:B\times\C^{N}\to E^{2}$ are complex epimorphisms.
Because the bicomplex inner product is componentwise
(Definition~\textup{\ref{def:inner-product}}), the orthogonal complement
satisfies
\begin{equation*}
\ker(\phi_{x})^{\perp}
=\ker(\phi^{1}_{x})^{\perp}\mathbf{e}\oplus\ker(\phi^{2}_{x})^{\perp}\mathbf{e}^{\dagger}.
\end{equation*}
Consequently, the classifying map decomposes as
\begin{equation*}
\iota_{\phi}
=(\iota_{\phi^{1}},\iota_{\phi^{2}})
:B\to\operatorname{Gr}_{m}(\mathbb{C}^{N})\times\operatorname{Gr}_{m}(\mathbb{C}^{N}),
\end{equation*}
which is precisely the idempotent decomposition of\/
$\operatorname{Gr}_{m}(\mathbb{B}\mathbb{C}^{N})$.
\end{remark}

\section{Leray--Hirsch theorem}\label{sec:Leray-hirsch}
In this section, we introduce the analogue of Chern classes $\wb_k(E)$ of a bicomplex fiber bundle $E$. We show that these characteristic classes are not determined by the usual complex Chern classes of the underlying complex bundles $E_\C$, but they satisfy similar properties. As in the complex case, the construction is based on the Leray--Hirsch theorem applied to the bicomplex projective bundle.

\subsection{Bicomplex projective spaces and their cohomology}

Recall that $\Gr_1(\C^n) \cong \CP^{n-1}$. By the idempotent decomposition of the Grassmannian, we define the bicomplex projective space as
\begin{equation*}
    \BCP^{n-1} = \Gr_1(\BC^n) \cong \Gr_1(\C^n)\times \Gr_1(\C^n)\cong\CP^{n-1} \times \CP^{n-1}.
\end{equation*}

Passing to the limit $n \to \infty$, we have
\begin{equation*}
    \BCP^\infty \cong \CP^\infty \times \CP^\infty.
\end{equation*}

Let $p_1, p_2: \BCP^\infty \to \CP^\infty$ be the projections onto the two factors. Let $a \in H^2(\C P^\infty;\Z)$ be the generator corresponding to the orientation of $\CP^1 \cong \mathbb{S}^2$. Consider
\begin{equation*}
    x = p_1^*(a) \quad \text{and}\quad y = p_2^*(a).
\end{equation*}
The cohomology of the bicomplex projective space is computed via the Künneth formula.

\begin{prop}\label{prop:cohomology_bcp}
    The integral cohomology of $\BCP^{n-1}$ is given by
    \begin{equation*}
        H^*(\BCP^{n-1};\Z) \cong \Z[x,y] / \langle x^{n}, y^{n} \rangle.
    \end{equation*}
    Consequently,
    \begin{equation*}
        H^*(\BCP^\infty;\Z) \cong \Z[x,y].
    \end{equation*}
\end{prop}

\begin{proof}
    Since $\BCP^{n-1} \cong \CP^{n-1} \times \CP^{n-1}$, the Künneth formula gives
    \begin{equation*}
        H^*(\BCP^{n-1};\Z) \cong H^*(\CP^{n-1};\Z) \otimes H^*(\CP^{n-1};\Z).
    \end{equation*}
    The cohomology of complex projective space is $H^*(\CP^{n-1};\Z) \cong \Z[x]/\langle x^{n}\rangle$, and similarly for $y$. The tensor product yields the stated ring. The vanishing of the $\Tor$-terms follows from the fact that each $H^k(\CP^{n-1};\Z)$ is a free $\Z$-module.
\end{proof}

Let $\pi: E \to B$ be a bicomplex fiber bundle of rank $n$. Consider the \emph{bicomplex projective bundle}
\begin{equation*}
    \P\pi: \P E \to B,
\end{equation*}
whose fiber over $x \in B$ is $\P(E_x) \cong \BCP^{n-1}$. By the idempotent decomposition, we have
\begin{equation*}
    \P\pi = (\P\pi^1, \P\pi^2): \P E^1\e \oplus \P E^2\ed \to B.
\end{equation*}

\begin{remark}\label{rem:projective_abuse}
    We are abusing notation with $\P$. Note that the bicomplex projective bundle differs from the complex projective bundle of the underlying complex vector bundle. Indeed, if $\pi:E \to B$ is viewed as a complex vector bundle of rank $2n$ (with the underlying complex structure) that we shall denote by $\pi:E_\C\to B$, then its complex projective bundle $\P E_{\C} \to B$ has fibers isomorphic to $\CP^{2n-1}$. In contrast, the bicomplex projective bundle $\P E$ has fibers isomorphic to
    \begin{equation*}
        \BCP^{n-1} \cong \CP^{n-1} \times \CP^{n-1},
    \end{equation*}
    which is a product of two projective spaces of dimension $n-1$, not a single projective space of dimension $2n-1$. This will be crucial in our definition of characteristic classes for the bicomplex case.
\end{remark}

The pullback bundle $(\P\pi)^*(E)$ contains a natural bicomplex subbundle of rank 1, the bicomplex tautological line bundle
\begin{equation*}
    L(\pi) = \{ (v, \ell) \in E \times \P E \mid v \in \ell \}.
\end{equation*}
Any bicomplex line $\ell \in \BCP^{n-1}$ decomposes as $\ell = \ell^1\e + \ell^2\ed$ with $\ell^1, \ell^2 \in \CP^{n-1}$. Hence $L(\pi)$ decomposes as
\begin{equation*}
    L(\pi) = L^1(\pi)\e \oplus L^2(\pi)\ed,
\end{equation*}
where $L^1(\pi)$ and $L^2(\pi)$ are complex line bundles over $\P E$.

\subsection{The Leray--Hirsch theorem for bicomplex fiber bundles}

By Theorem~\ref{thm:classification}, isomorphism classes of bicomplex line bundles $L \to B$ are in bijection with homotopy classes of maps
\begin{equation*}
    f: B \to \BCP^\infty.
\end{equation*}

\begin{defn}
    The \emph{first bicomplex characteristic class} of a line bundle $L \to B$ is defined as
    \begin{equation}
        \wb_1(L) = f^*(x + y) = f^*(x) + f^*(y) \in H^2(B;\Z),
    \end{equation}
    where $f: B \to \BCP^\infty$ is a classifying map for $L$.
\end{defn}

\begin{remark}
    The class $x + y$ in $H^2(\BCP^\infty;\Z)$ is characterized by the commutativity of the factors and the canonical normalization. In terms of complex Chern classes, if $L = L^{1}\e \oplus L^{2}\ed$ is viewed as a complex vector bundle of rank $2$ (via the underlying complex structure), then the complex first Chern class satisfies
    \begin{align*}
        c_1(L^{1}) + c_{1}(L^{2}) &=
        f^*(p_1^*(a)) + f^*(p_2^*(a))\\
        &=f^*(x) + f^*(y)\\
        &=\wb_1(L).\\
    \end{align*}
     However, the complex second Chern class $c_2(L)$ is not captured by any bicomplex Chern class.
\end{remark}

Consider $f:\P E\to \BCP^\infty$, the classifying map given by Theorem~\ref{thm:classification}. Let $\xi_x = f^*(x)$ and $\xi_y = f^*(y)$. The following is the bicomplex analogue of the classical Leray--Hirsch theorem.

\begin{thm}[Leray--Hirsch for bicomplex fiber bundles]\label{thm:leray_hirsch}
    Let $\pi:E \to B$ be a bicomplex fiber bundle of rank $n$. The cohomology ring $H^*(\P E;\Z)$ is a free $H^*(B;\Z)$-module with basis
    \begin{equation}\label{eq:basis}
        \mathcal{B} = \{ \xi_x^i \xi_y^j \mid 0 \le i \le n-1,\ 0 \le j \le n-1 \}.
    \end{equation}
\end{thm}

\begin{proof}[Sketch of proof]
    For the trivial bundle $E = B \times \BC^n$, we have $\mathbb{P}E \cong B \times \BCP^{n-1}$. By the Künneth formula and Proposition~\ref{prop:cohomology_bcp},
    \begin{equation*}
        H^*(\mathbb{P}E;\Z) \cong H^*(B;\Z) \otimes \Z[x,y]/\langle x^n, y^n\rangle,
    \end{equation*}
    which is exactly the statement with basis $\{x^i y^j \mid 0 \le i,j \le n-1\}$.

    For a general bicomplex bundle, the proof follows the classical argument: one covers $B$ by trivializing open sets, uses the Mayer--Vietoris sequence, and compares the cohomology rings. The idempotent decomposition ensures that the argument decomposes componentwise, reducing to the complex case.
\end{proof}

The following result is essential for computations.

\begin{prop}[Splitting principle]\label{prop:splitting}
    Let $\pi:E \to B$ be a bicomplex fiber bundle. Then there exists a continuous map $f: P \to B$ such that
    \begin{enumerate}
        \item $f^*E$ is a direct sum of bicomplex line bundles on $P$;
        \item $f^*: H^*(B;\Z) \to H^*(P;\Z)$ is injective.
    \end{enumerate}
    The map $f$ is called a \emph{splitting map}.
\end{prop}

\begin{proof}
    The proof is identical to the complex case. We proceed by induction on the rank $n$. For $n=1$, take $P=B$ and $f=\mathrm{id}$. Assume the result holds for ranks $< n$. Let $\P\pi: \P E \to B$ be the bicomplex projective bundle. Over $\P E$, we have the tautological line subbundle $L(\pi) \subset \pi^*E$. Consider the idempotent decomposition
\begin{equation*}
\pi^*E=\pi^*E^1e\oplus\pi^*E^2e^\dagger.
\end{equation*}
The tautological line bundle decomposes accordingly as
\begin{equation*}
L(\pi)=L^1(\pi)e\oplus L^2(\pi)e^\dagger.
\end{equation*}
Taking the orthogonal complements of $L^1(\pi)$ in $\pi^*E^1$ and of $L^2(\pi)$ in $\pi^*E^2$, we obtain complex subbundles $Q^1$ and $Q^2$. Hence,
\begin{equation*}
Q=Q^1e\oplus Q^2e^\dagger
\end{equation*}
is a bicomplex subbundle, and we obtain the splitting sequence
\begin{equation*}
0\longrightarrow L(\pi)\longrightarrow \pi^*E\longrightarrow Q\longrightarrow 0,
\end{equation*}
Moreover,
\begin{equation*}
\pi^*E=L(\pi)\oplus Q.
\end{equation*}

The bundle $Q$ has rank $n-1$. By the induction hypothesis, there exists a map $g: P' \to \P E$ such that $g^*Q$ splits into line bundles and $g^*$ is injective on cohomology. Let $f = \P\pi \circ g: P' \to B$. Then $f^*E = g^*\pi^*E \cong g^*L(\pi) \oplus g^*Q$, where $g^*L(\pi)$ is a bicomplex line bundle and $g^*Q$ splits by induction. The injectivity of $f^*$ follows from the injectivity of $\P\pi^*$ (Leray--Hirsch) and of $g^*$.
\end{proof}

\subsection{Bicomplex characteristic classes}

Since $\xi_x^n + \xi_y^n$ can be expressed as a linear combination of the basis elements in \eqref{eq:basis} with coefficients in $H^*(B;\Z)$, there is a unique relation of the form
\begin{equation}\label{eq:linear-relation}
    \xi_x^n + \xi_y^n + \sum\limits_{\substack{0\leq i,j\leq n-1\\ i+j\leq n}} (\mathbb{P}\pi)^*(\alpha_{i,j})\xi_x^{i}\xi_y^{j} = 0
\end{equation}
in $H^*(\P E;\Z)$, where $\alpha_{i,j} \in H^{2(n-i-j)}(B;\Z)$.

\begin{defn}\label{def:bicomplex-chiern-classes}
    For $k = 1, \dots, n$, the \emph{$k$-th bicomplex Chern class} of $E$ is defined as
    \begin{equation*}
        \wb_k(E) = \sum_{i+j = n-k} \alpha_{i,j} \in H^{2k}(B;\Z).
    \end{equation*}
    We set $b_0=2$, then the \emph{total bicomplex characteristic class} to be
    \begin{equation*}
        \wb(E) = b_0 + \wb_1(E) + \cdots + \wb_n(E) \in H^*(B;\Z).
    \end{equation*}
\end{defn}

\begin{remark}
One might alternatively try to define the bicomplex Chern classes by imposing the relation in the Leray--Hirsch theorem for $(\xi_x+\xi_y)^n$. Expanding this expression yields
\begin{equation*}
(\xi_x+\xi_y)^n \;=\; \sum_{i=0}^{n}\binom{n}{i}\,\xi_x^{\,n-i}\xi_y^{\,i}.
\end{equation*}
For $1\le i\le n-1$, the cross--terms $\xi_x^{\,n-i}\xi_y^{\,i}$ already belong to the basis 
$\mathcal{B}=\{\xi_x^{p}\xi_y^{q}\mid 0\le p,q\le n-1\}$ 
of $H^{*}(\mathbb{P}E;\Z)$ as an $H^{*}(B;\Z)$--module. Hence they contribute directly to the coefficients of total fibre--degree $n$ in the relation in Leray--Hirch theorem. The sum of all these coefficients is
\begin{equation*}
\sum_{i=0}^{n}\binom{n}{i}=2^{n},
\end{equation*}
which would force $b_0=2^{n}$, a value that depends on the rank $n$ and is incompatible with the idempotent decomposition. Indeed, for a bundle $E=E^{1}\mathbf{e}\oplus E^{2}\ed$ we must have
\begin{equation*}
b_0(E)=c_0(E^{1})+c_0(E^{2})=1+1=2.
\end{equation*}
Therefore, the relation $\xi_x^{n}+\xi_y^{n}$ is the correct choice: its degree--$n$ terms have coefficients $1+1=2$, yielding the topologically consistent normalization $b_0=2$ independent of the rank.
\end{remark}

\begin{remark}
Since $E = E^1\e\oplus E^2\ed$ can be viewed as a complex vector bundle of rank $2n$, its total complex Chern class is
\begin{equation*}
c(E_{\C}) = c(E^1)c(E^2).
\end{equation*}
However, the bicomplex Chern classes $b_k(E)$ differ from the complex ones $c_k(E_{\C})$ for the underlying complex vector bundle $E_\C$. In particular, $\wb_{2n}(E) = 0$ by definition (since $\wb_k$ is defined only for $k \le n$), whereas $c_{2n}(E_{\C})$ may be nontrivial.
\end{remark}

\subsection{Properties of bicomplex characteristic classes}
In this subsection, we prove that bicomplex Chern classes can be obtained from complex Chern classes.  We also note that there is no analogous result for complex whitney formula.

Let 
\begin{equation*}
p_1: \mathbb{P}E \cong \P E^1\e \oplus \P E^2 \ed\to \mathbb{P}E^1
\end{equation*}
and 
\begin{equation*}
p_2: \P E \cong \P E^1\e \oplus \P E^2\ed \to \mathbb{P}E^2
\end{equation*}
be the canonical projections. One checks that
\begin{equation*}\label{eq:xi-complex}
p_1^*(\xi(E^1)) = \xi_x(E) \quad \text{and}\quad p_2^*(\xi(E^2)) = \xi_y(E),
\end{equation*}
where $\xi(E^1)$ and $\xi(E^2)$ are the classes that define the complex Chern classes $c_k(E^1)$ and $c_k(E^2)$, respectively. We have the following 
\begin{thm}\label{prop:bicomplex-class-descomposotion}
    If $E=E^1\e\oplus E^2\ed$, then 
    \begin{equation*}
        \wb_k(E)=c_k(E^1)+c_k(E^2) \in H^{2k}(B;\Z)\quad \text{for $k=0,\cdots,n$},
    \end{equation*}
    where $c_k(E^1)$ and $c_k(E^2)$ are the complex Chern classes.
\end{thm}

\begin{proof}
Note that $\xi(E^1)$ and $\xi(E^2)$ in Equation \eqref{eq:xi-complex} satisfy the complex Leray--Hirsch theorem:
\begin{gather*}
    \xi(E^1)^n+\sum_{k=1}^n (\P\pi^1)^*(c_k(E^1))\xi(E^1)^{n-k}=0,\\
    \xi(E^2)^n+\sum_{k=1}^n (\P\pi^2)^*(c_k(E^2))\xi(E^2)^{n-k}=0.
\end{gather*}
where $(\P\pi^i)^*:H^*(B;\Z)\to H^*(PE^i;\Z)$ and $c_k(E^1),c_k(E^2)\in H^{2i}(B;\Z)$. We have the following commutative diagram:
\begin{equation*}
    \xymatrix{\P E\ar[r]^{p_1}\ar[d]_{\pi}&\P E^1\ar[d]^{\pi^1}\\ B\ar@{=}[r]&B.}
\end{equation*}
Then, $(p^1)^*\circ(\pi^1)^*=(\pi^1\circ p_1)^*=\pi^*$. The same for $(p^2)^*\circ(\pi^2)^*=(\pi^2\circ p_2)^*=\pi^*$. Substituting in Leray--Hirsch relation in Equation~\eqref{eq:linear-relation}, we get
\begin{align*}
    \xi_x^n + \xi_y^n &=p^*_1(\xi(E^1))^n+p^*_2(\xi(E^2))^n\\
    &=-\sum_{k=1}^n (\P\pi)^*(c_k(E^1))\xi(E^1)^{n-k}-\sum_{k=1}^n (\P\pi)^*(c_k(E^2))\xi(E^2)^{n-k}.
\end{align*}
Since relation in Equation~\eqref{eq:linear-relation} is unique, we have that $\wb_k=c_k(E^1)+c_k(E^2)\in H^{2k}(B;\Z)$.
\end{proof}

\begin{example}\label{ex:trivial-component}
Let $E^1 \to B$ be any complex vector bundle of rank $n$, and let $E^2 = \underline{\C}^n \to B$ be the trivial complex vector bundle of the same rank. Consider the bicomplex fiber bundle
\begin{equation*}
E = E^1\e \oplus \underline{\C}^n\ed.
\end{equation*}

Using $\xi_y^n = 0$, we now could compute the bicomplex characteristic classes of $E$ using the Leray-Hirsch Theorem~\ref{thm:leray_hirsch}, since every term containing a factor of $\xi_y^n$ vanishes. But by the previous Proposition~\ref{prop:bicomplex-class-descomposotion} 
\begin{equation*}
    \wb_k(E)=c_k(E^1)+c_k(\underline{\C})=c_k(E^1)
\end{equation*}

Thus, in this special case where one of the idempotent components is trivial, the bicomplex Chern classes recover exactly the classical Chern classes of the non-trivial complex component. In particular, the theory of complex characteristic classes is embedded in the bicomplex theory as the subfamily of bundles with a trivial $\ed$-component (or $\e$-component).

This example also illustrates that the non-triviality of the bicomplex classes arises precisely from the interaction between two non-trivial complex components. When one component is trivial, no such interaction occurs, and the theory reduces to the complex case.
\end{example}

\begin{thm}\label{thm:bicomplex_chern_properties}
    For each bicomplex fiber bundle $E \to B$, the bicomplex Chern classes and the bicomplex total Chern class in $H^*(B;\Z)$ satisfies: If $g: B' \to B$ is a continuous map, then
        \begin{equation*}
            g^*(\wb_k(E)) = \wb_k(g^*E).
        \end{equation*} 
    \end{thm}
    
\begin{proof}
    Consider the pullback $E' = g^*E$. The classes $\xi_x(E)$ and $\xi_y(E)$ pull back to $\xi_x(E')$ and $\xi_y(E')$ under the induced map $\P(g): \P E' \to \P E$. Applying $g^*$ to the relation of Equation~\eqref{eq:linear-relation} gives the corresponding relation for $E'$, and thus $g^*(b_k(E)) = b_k(E')$.
\end{proof}

\begin{remark}[Failure of the Whitney sum formula]\label{rem:whitney-failure}
The Whitney sum formula $\wb(E \oplus F) = \wb(E)\cdot \wb(F)$ does not hold for the bicomplex characteristic classes. To see this, let $B = \mathbb{CP}^2$ and let $L \to \mathbb{CP}^2$ be the tautological complex line bundle with $a = c_1(L)$ and $a^2 \neq 0$. Define two bicomplex line bundles:
\begin{equation*}
L_1 = L\e \oplus \underline{\C}\ed \quad
\text{and}\quad L_2 = \underline{\C}\e \oplus L\ed.
\end{equation*}

Their first bicomplex Chern classes are
\begin{equation*}
\wb_1(L_1) = c_1(L) + c_1(\underline{\C}) = a + 0 = a,
\end{equation*}
\begin{equation*}
\wb_1(L_2) = c_1(\underline{\C}) + c_1(L) = 0 + a = a.
\end{equation*}
Hence their total classes are
\begin{equation*}
\wb(L_1) = 2 + a, \quad \wb(L_2) = 2 + a.
\end{equation*}

Now consider the sum $E = L_1 \oplus L_2$. Its complex components are
\begin{equation*}
E^1 = L_1^1 \oplus L_2^1 = L \oplus \underline{\C}, \quad\text{and}\quad
E^2 = L_1^2 \oplus L_2^2 = \underline{\C} \oplus L.
\end{equation*}
Both components are isomorphic to $L \oplus \underline{\C}$. Over the projective bundle $\P E \cong \P E^1\e \oplus \mathbb{P}E^2\ed$, the classes $\xi_x$ and $\xi_y$ satisfy the relations:
\begin{equation*}
\xi_x^2 + a\xi_x = 0 \quad \text{and}\quad \xi_y^2 + a\xi_y = 0.
\end{equation*}
Then
\begin{equation*}
\xi_x^2 + \xi_y^2 = -a(\xi_x + \xi_y).
\end{equation*}

Comparing with the general Leray-Hirsch relation
\begin{equation*}
\xi_x^2 + \xi_y^2 + \alpha_{0,0} + \alpha_{1,0}\xi_x + \alpha_{0,1}\xi_y + \alpha_{1,1}\xi_x\xi_y = 0,
\end{equation*}
we obtain
\begin{equation*}
\alpha_{1,0} = a, \quad \alpha_{0,1} = a, \quad \alpha_{1,1} = 0, \quad \alpha_{0,0} = 0.
\end{equation*}

The bicomplex characteristic classes of $E$ are
\begin{equation*}
\wb_1(E) = \alpha_{1,0} + \alpha_{0,1} = 2a,
\end{equation*}
\begin{equation*}
\wb_2(E) = \alpha_{0,0} = 0.
\end{equation*}
Therefore,
\begin{equation*}
\wb(E) = 2 + 2a.
\end{equation*}

However,
\begin{equation*}
b(L_1)\cdot b(L_2) = (2 + a)^2 = 4 + 4a + a^2.
\end{equation*}
Clearly,
\begin{equation*}
b(E) \neq b(L_1)\cdot b(L_2).
\end{equation*}
Thus the Whitney sum formula fails already for a direct sum of two bicomplex line bundles.
\end{remark}

\begin{remark}
In the complex case the following holds: if $E\cong\bigoplus_{i=1}^{n}L_{i}$ is a sum of complex line bundles and $E$ admits a nowhere vanishing section, then the top Chern class vanishes; equivalently,
\begin{equation*}
c_{1}(L_{1})\smile\cdots\smile c_{1}(L_{n})=0.
\end{equation*}
No analogous result holds for regular sections of bicomplex bundles. Example~\ref{ex:regular_from_zero_divisors} provides a concrete illustration: the bicomplex bundle
\begin{equation*}
E=\bigl((p_{1}^{*}\gamma_{\mathbb{C}})\mathbf{e}\oplus\underline{\mathbb{C}}\,\mathbf{e}^{\dagger}\bigr)
   \oplus
   \bigl(\underline{\mathbb{C}}\,\mathbf{e}\oplus(p_{2}^{*}\gamma_{\mathbb{C}})\mathbf{e}^{\dagger}\bigr)
\end{equation*}
over $B=\CP^1\times\CP^1$ admits the regular section
\begin{equation*}
s=(0\e+1\ed,1\e+0\ed),
\end{equation*}
yet
\begin{equation*}
b_{1}(L_1)\smile b_1(L_2)=x\smile y\neq 0\in H^{4}(B;\mathbb{Z}),
\end{equation*}
where $x=p_{1}^{*}(a)$ and $y=p_{2}^{*}(a)$ are the degree-$2$ generators coming from the two factors. Thus the existence of a regular section does not force the vanishing of the product of the first bicomplex Chern classes of the summands.
\end{remark}

The preceding example shows that the existence of a regular section in a bicomplex bundle $E$ is not governed by the vanishing of any product of the complex Chern classes of its idempotent components $E^{1}$ and $E^{2}$.

\section{Universal bicomplex Chern classes}

In this section, we study the universal  classifying space for $\GL(n,\BC)$. The idempotent decomposition of the Lie groups involved allows us to give a definition in terms of classifying space of $\GL(n,\C)$. Recall from Section~\ref{sec:bicomplex-modules} that the general linear group and the unitary group decompose as
\begin{equation*}
    \GL(n,\BC) \cong \GL(n,\C)\e \oplus \GL(n,\C)\ed\quad\text{and}\quad
    \U(n,\BC) \cong \U(n,\C)\e \oplus \U(n,\C)\ed.
\end{equation*}

These are isomorphisms of Lie groups. Consequently, the classifying spaces also decompose as products.

\begin{prop}\label{prop:classifying_spaces}
    The idempotent decomposition induces homeomorphisms
    \begin{equation*}
        \B\!\GL(n,\BC) \cong \B\!\GL(n,\C) \times \B\!\GL(n,\C),
    \end{equation*}
    and
    \begin{equation*}
        \B\!\U(n,\BC) \cong \B\!\U(n,\C) \times \B\!\U(n,\C).
    \end{equation*}
\end{prop}

\begin{proof}
    The functoriality of the classifying space construction sends products of topological groups to products of classifying spaces. Since the idempotent decomposition is a product of Lie groups at the level of points, the result follows.
\end{proof}

By functoriality of the classifying space, the involution $\iota$ in Equation~\ref{eq:involutionmat} induces an involution map
\begin{equation}\label{eq:involution-esp}
\begin{gathered}
\tau:\B\!\GL(n,\BC)\to\B\!\GL(n,\BC)\\
([f_{1}],[f_{2}])\mapsto([f_{2}],[f_{1}]),
\end{gathered}
\end{equation}
on $\B\!\GL(n,\BC)\cong\B\!\GL(n,\C)\times\B\!\GL(n,\C)$, that interchanges factors.

\begin{prop}\label{prop:cohomology_classifying}
    The cross product gives an isomorphism
    \begin{equation*}
        H^*(\B\!\GL(n,\C);\Z) \otimes H^*(\B\!\GL(n,\C);\Z)
        \cong H^*(\B\!\GL(n,\BC);\Z).
    \end{equation*}
\end{prop}

\begin{proof}
    By Proposition~\ref{prop:classifying_spaces}, $\B\!\GL(n,\BC) \cong \B\!\GL(n,\C) \times \B\!\GL(n,\C)$. The Künneth formula (\cite[Theorem~3.18]{hatcher}) applies since $H^k(\B\!\GL(n,\C);\Z)$ is a finitely generated free $\Z$-module for each $k$ (indeed, $H^*(\B\!\GL(n,\C);\Z) \cong \Z[c_1,\dots,c_n]$ with $c_k \in H^{2k}$). Hence the torsion terms vanish, and the stated isomorphism holds.
\end{proof}

Let $p_1, p_2: \B\!\GL(n,\BC) \to \B\!\GL(n,\C)$ be the projections onto the first and second factors. Explicitly, if $c_1,\dots,c_n$ denote the universal complex Chern classes, then
\begin{equation}
    H^*(\B\!\GL(n,\BC);\Z)
    \cong \Z[p_1^*(c_1), \dots, p_1^*(c_n), \, p_2^*(c_1), \dots, p_2^*(c_n)].
\end{equation}

\subsection{Reduction to the unitary group}

A crucial fact is that the unitary group is a deformation retract of the general linear group, even in the bicomplex setting.

\begin{prop}\label{prop:deformation_retract}
    The inclusion $\U(n,\BC) \hookrightarrow \GL(n,\BC)$ is a deformation retract. Consequently, we may use $\U(n,\BC)$ instead of $\GL(n,\BC)$ to define the universal bicomplex Chern classes.
\end{prop}

\begin{proof}
    This is a consequence of the existence of a bicomplex Gram--Schmidt process given in Subsection~\ref{subsec:gram-schmidt}. By the idempotent decomposition, any matrix $A \in \GL(n,\BC)$ decomposes as $A = A^1\e + A^2\ed$ with $A^1, A^2 \in \GL(n,\C)$. Applying the complex Gram--Schmidt process to $A^1$ and $A^2$ independently yields unitary matrices $U^1 \in \U(n,\C)$ and $U^2 \in \U(n,\C)$, and hence a bicomplex unitary matrix $U = U^1\e + U^2\ed \in \U(n,\BC)$. The standard deformation retraction from $\GL(n,\C)$ to $\U(n,\C)$ applies componentwise, giving the desired homotopy.
\end{proof}

Thus, in the sequel, we may freely replace $\GL(n,\BC)$ by $\U(n,\BC)$ when defining universal characteristic classes.

\subsection{Universal bicomplex Chern classes}

The idempotent decomposition of the tautological bundle $\gamma_\BC^n\to \Gr_n(\BC^\infty)$ provides two canonical complex subbundles $\gamma_\C^n\e$ and $\gamma_\C^n\ed$ of rank $n$. Let us denote by $c_k^{\mathrm{univ}}$ the $k$-th universal complex Chern class.

\begin{defn}\label{def:universal-b}
The \emph{universal bicomplex Chern classes} are
\begin{equation*}
\mathbf{b}_k= c_k(\gamma\e) \;+\; c_k(\gamma\ed)
\;\in\; H^{2k}(\Gr_n(\BC^\infty);\Z),\qquad k=1,\dots,n.
\end{equation*}
\end{defn}

These classes satisfy the following natural properties, which are immediate from the definition and the corresponding properties of ordinary Chern classes:

\begin{prop}\label{prop:universal-axioms}
The classes $\mathbf{b}_1,\dots,\mathbf{b}_n$ satisfy:
\begin{enumerate}
\item\label{it:nat} \textbf{Naturality.} For every bicomplex bundle $E\to B$ of rank $n$ with classifying map $f:B\to \Gr_n(\BC^\infty)$,
\begin{equation*}
b_k(E)=f^*\mathbf{b}_k.
\end{equation*}
\item\label{it:nor} \textbf{Normalisation.} For the universal bicomplex line bundle over $\mathbb{BCP}^\infty$,
\begin{equation*}
\mathbf{b}=2+x+y\in H^2(\mathbb{BCP}^\infty;\Z)\cong\Z[x,y].
\end{equation*}
\item \label{it:inv}\textbf{$\tau$--invariance.} If $\tau$ is the involution map in Equation~\ref{eq:involution-esp}, then $\tau^*\mathbf{b}_k=\mathbf{b}_k$ for all $k$.
\item\label{it:red} \textbf{Reduction to one factor.} With $i_1,i_2:\Gr_n(\C^\infty)\hookrightarrow \Gr_n(\BC^\infty)$ the inclusions of the two factors,
\begin{equation*}
i_1^*\mathbf{b}_k=c_k^{\mathrm{univ}}=i_2^*\mathbf{b}_k.
\end{equation*}
\item\label{it:diag} \textbf{Diagonal restriction.} With $\Delta:\operatorname{Gr}_n(\C^\infty)\to \Gr_n(\BC^\infty)$ the diagonal embedding,
\begin{equation*}
\Delta^*\mathbf{b}_k=2\,c_k^{\mathrm{univ}}.
\end{equation*}
\end{enumerate}
\end{prop}

\begin{remark}\label{rem:axioms_not_unique}
Properties \eqref{it:nat}--\eqref{it:diag} do not characterize the classes $\mathbf{b}_{k}$ uniquely.
To see this, write  $a_{k}=p_{1}^{*}(c_{k}^{\mathrm{univ}})$ and $b_{k}=p_{2}^{*}(c_{k}^{\mathrm{univ}})$ in $H^{*}\bigl(\B\!\GL(n,\BC);\Z)$, which are the pullbacks of the universal complex Chern classes via the two projections
$p_{1},p_{2}\colon\B\!\GL(n,\B\C)\to\B\!\GL(n,\C)$.
For $n\geq 2$, set
\begin{equation*}
Q=a_{1}a_{2}b_{1}+a_{1}b_{1}b_{2}-a_{1}^{2}b_{2}-a_{2}b_{1}^{2}
\in H^{8}\bigl(\operatorname{BGL}(n,\mathbb{B}\mathbb{C});\mathbb{Z}\bigr).
\end{equation*}
Notice that $\Delta^*(a_1)=\Delta^*(b_1)=c_1$ and $\Delta^*(a_2)=\Delta^*(b_2)=c_2$, which gives $\Delta^*(Q)=0$. Then the modified class $\mathbf{b}_{4}'=\mathbf{b}_{4}+Q$ also satisfies properties \eqref{it:nat}--\eqref{it:diag}, yet $\mathbf{b}_{4}'\neq\mathbf{b}_{4}$.
\end{remark}

\section{The bicomplex Chern character}\label{sec:chern-char}

In this section we construct the bicomplex analogue of the classical Chern character.
A naive attempt to define it via the splitting principle and Newton polynomials in the classes $b_k(E)$ runs into an essential obstruction: the $b_k$'s are not the elementary symmetric functions of any invariant system of ``roots'' associated to a bicomplex splitting.  We shall instead define the character directly from the idempotent decomposition, which yields a natural additive invariant.  We also give its multiplicative lift to cohomology with bicomplex coefficients.

\subsection{Idempotent decomposition of the splitting principle}\label{subsec:genuine-splitting}

Let $\pi:E\to B$ be a bicomplex vector bundle of rank $n$. Consider the canonical idempotent decomposition $E=E^1\e\oplus E^2\ed$. A bicomplex flag in $E$
\begin{equation*}
0=F_{0}\subset F_{1}\subset\cdots\subset F_{n}=E
\end{equation*}
with $\operatorname{rank}_{\BC}F_{i}=i$, has an idempotent decomposition into a pair of ordinary complex flags
\begin{equation*}
0=F^{1}_{0}\subset F^{1}_{1}\subset\cdots\subset F^{1}_{n}=E^1,\quad
0=F^{2}_{0}\subset F^{2}_{1}\subset\cdots\subset F^{2}_{n}=E^2
\end{equation*}
with the same indexing.  The associated bundle of bicomplex flags is therefore 
\begin{equation*}
\operatorname{Fl}_{\BC}(E)=\operatorname{Fl}(E^1)\times_B\operatorname{Fl}(E^2).
\end{equation*}
Over $\operatorname{Fl}_{\BC}(E)$ the bundle $E$ splits as a sum of bicomplex line bundles
\begin{equation*}
E\cong\bigoplus_{i=1}^{n}M_{i},\quad\text{with}\quad M_{i}=L^{1}_{i}\e\oplus L^{2}_{i}\ed,
\end{equation*}
where $L^{1}_{i}=F^{1}_{i}/F^{1}_{i-1}$ and $L^{2}_{i}=F^{2}_{i}/F^{2}_{i-1}$.  The first bicomplex Chern classes
\begin{equation*}
t_{i}=b_{1}(M_{i})=c_{1}(L^{1}_{i})+c_{1}(L^{2}_{i})\in H^{2}\bigl(\operatorname{Fl}_{\BC}(E);\mathbb Z\bigr)
\end{equation*}
are invariant under the diagonal action of the symmetric group $S_{n}$ (simultaneous permutation of the two complex flags), but not under the full product $S_{n}\times S_{n}$.  Consequently, the unordered collection $\{t_{1},\dots,t_{n}\}$ is not a well-defined invariant of $E$; the elementary symmetric functions $e_{k}(t_{1},\dots,t_{n})$ depend on the choice of simultaneous splitting and do not descend to $H^{2k}(B;\mathbb Z)$ in general.  This explains why the $b_{k}(E)$ cannot be obtained from the $t_{i}$ via the classical Newton identities.

\subsection{The bicomplex Chern character}

We now define two versions of the Chern character.  The first takes values in ordinary rational cohomology; the second in cohomology with bicomplex coefficients and is multiplicative with respect to the bicomplex tensor product.

\subsubsection*{The ordinary bicomplex Chern character}

\begin{defn}\label{def:ch-ordinary}
Let $E\to B$ be a bicomplex vector bundle of rank $n$.  Its \emph{bicomplex Chern character} is
\begin{equation*}
\operatorname{ch}_{\BC}(E)
=\operatorname{ch}(E^1)+\operatorname{ch}(E^2)
\;\in\;H^{\mathrm{ev}}(B;\mathbb Q),
\end{equation*}
where $H^{\mathrm{ev}}(B;\mathbb Q)$ is the rational even-degree cohomology ring of $B$.
\end{defn}

Since the idempotent decomposition is canonical, $\operatorname{ch}_{\BC}(E)$ is well-defined and natural with respect to pullbacks.  Expanding the two classical characters, we obtain
\begin{equation*}
\operatorname{ch}_{\BC}(E)
=2n+\sum_{k\ge 1}\frac{1}{k!}\bigl(p_{k}(E^1)+p_{k}(E^2)\bigr),
\end{equation*}
where $p_{k}$ denotes the $k$-th Newton polynomial in the ordinary Chern classes.  In particular, the component of degree $2k$ is
\begin{equation*}
\operatorname{ch}_{\BC}(E)_{2k}
=\frac{1}{k!}\bigl(c_{1}(E^1)^{k}+\cdots+c_{k}(E^1)+\cdots\bigr)
+\frac{1}{k!}\bigl(c_{1}(E^2)^{k}+\cdots+c_{k}(E^2)+\cdots\bigr).
\end{equation*}

\begin{prop}[Properties of $\operatorname{ch}_{\BC}$]\label{prop:ch-properties}
Let $E,F\to B$ be bicomplex vector bundles.
\begin{enumerate}
\item\label{it:chern-add} \textbf{Additivity.} $\operatorname{ch}_{\BC}(E\oplus F)=\operatorname{ch}_{\BC}(E)+\operatorname{ch}_{\BC}(F)$.
\item\label{it:chern-nat} \textbf{Naturality.} For any continuous map $g:B'\to B$,
$g^{*}\operatorname{ch}_{\BC}(E)=\operatorname{ch}_{\BC}(g^{*}E)$.
\item\label{it:chern-norm} \textbf{Normalisation for line bundles.} If $L=L^{1}\e\oplus L^{2}\ed$ is a bicomplex line bundle, then
\begin{equation*}
\operatorname{ch}_{\BC}(L)=e^{c_{1}(L^{1})}+e^{c_{1}(L^{2})}.
\end{equation*}
\end{enumerate}
\end{prop}

\begin{proof}
The proof of \eqref{it:chern-add} follows from $(E\oplus F)\e=E^1\e\oplus F^1\e$ and the additivity of the classical Chern character. Meanwhile \eqref{it:chern-nat} is immediate because the idempotent decomposition commutes with pullback. And \eqref{it:chern-norm} is the rank-one case of the definition.
\end{proof}

\begin{remark}[Failure of multiplicativity]\label{rem:non-mult}
In general,
\begin{equation*}
\operatorname{ch}_{\BC}(E\otimes_{\BC}F)\neq\operatorname{ch}_{\BC}(E)\cdot\operatorname{ch}_{\BC}(F).
\end{equation*}
Indeed, $(E\otimes_{\BC}F)\e\cong E^1\e\otimes_{\C} F
^1\ed$ and similarly for $\ed$, so the classical multiplicativity gives
\begin{equation*}
\operatorname{ch}_{\BC}(E\otimes_\BC F)
=\operatorname{ch}(E^1)\operatorname{ch}(F^1)+\operatorname{ch}(E^2)\operatorname{ch}(F^2),
\end{equation*}
whereas the product of the ordinary characters is
\begin{equation*}
\operatorname{ch}_{\BC}(E)\operatorname{ch}_{\BC}(F)
=\operatorname{ch}(E^1)\operatorname{ch}(F^1)
+\operatorname{ch}(E^1)\operatorname{ch}(F^2)
+\operatorname{ch}(E^2)\operatorname{ch}(F^1)
+\operatorname{ch}(E^2)\operatorname{ch}(F^2).
\end{equation*}
The cross terms $\operatorname{ch}(E^1)\operatorname{ch}(F^2)$ and $\operatorname{ch}(E^2)\operatorname{ch}(F^1)$ have no counterpart in the bicomplex tensor product.
\end{remark}

\subsubsection*{The multiplicative lift}

The obstruction in Remark~\ref{rem:non-mult} disappears if we remember the idempotent origin of the two summands.

\begin{defn}\label{def:ch-lift}
The \emph{bicomplex Chern character with coefficients} is
\begin{equation*}
\widetilde{\operatorname{ch}}_{\BC}(E)
=\operatorname{ch}(E^1)\e+\operatorname{ch}(E^2)\ed\in\;H^{\mathrm{ev}}(B;\BC\otimes\mathbb Q).
\end{equation*}
\end{defn}

Because $\BC\otimes\mathbb Q\cong\mathbb Q\e\oplus\mathbb Q\ed$ with componentwise multiplication, we have
\begin{equation*}
\widetilde{\operatorname{ch}}_{\BC}(E\otimes_{\BC}F)
=\widetilde{\operatorname{ch}}_{\BC}(E)\cdot\widetilde{\operatorname{ch}}_{\BC}(F),
\end{equation*}
and
\begin{equation*}
\widetilde{\operatorname{ch}}_{\BC}(E\oplus F)
=\widetilde{\operatorname{ch}}_{\BC}(E)+\widetilde{\operatorname{ch}}_{\BC}(F).
\end{equation*}
Thus $\widetilde{\operatorname{ch}}_{\BC}$ is a ring homomorphism from the bicomplex $K$-theory $K_{\BC}(B)$ to $H^{\mathrm{ev}}(B;\BC\otimes\mathbb Q)$.  The ordinary character of Definition~\ref{def:ch-ordinary} is recovered by the augmentation $\BC\otimes\mathbb Q\to\mathbb Q$, $\e,\ed\mapsto 1$.

\subsection{Comparison with the bicomplex Chern classes}

The bicomplex Chern classes $b_{k}(E)=c_{k}(E^1)+c_{k}(E^2)$ are, in a precise sense, the traces of the idempotent components.  The Chern character $\operatorname{ch}_{\BC}(E)$ is strictly stronger: it cannot be expressed as a polynomial (or even a convergent power series) in the classes $b_{1}(E),\dots,b_{n}(E)$ alone.

\begin{example}
Take $B=\C\P^{1}$ and let $\gamma_\C^1\to\C\P^{1}$ be the tautological line bundle with $c_{1}(L)=a\neq 0$.  Set
\begin{equation*}
E=\gamma_\C^1\e\oplus {\gamma_\C^1}^{*}\ed\quad \text{and}\quad
F=\underline{\C}\e\oplus\underline{\C}\ed.
\end{equation*}
Then $E^1=\gamma_\C^1$, $E^2={\gamma_\C^1}^{*}$, $F^1=F^2=\underline{\mathbb C}$.  Hence
\begin{equation*}
b_{1}(E)=c_{1}(\gamma_\C^1)+c_{1}({\gamma_\C^1}^{*})=a-a=0=b_{1}(F),
\end{equation*}
and $b_{k}(E)=b_{k}(F)=0$ for $k>1$.  However,
\begin{equation*}
\operatorname{ch}_{\BC}(E)=e^{a}+e^{-a}=2+a^{2}+\frac{a^{4}}{12}+\cdots\neq 2=\operatorname{ch}_{\BC}(F).
\qedhere
\end{equation*}
\end{example}

\begin{remark}
The same example shows that the ordinary bicomplex Chern character does not factor through the subring $\mathbb Z[b_{1},\dots,b_{n}]\subset H^{*}(B;\mathbb Z)$ generated by the bicomplex Chern classes.  The classes $b_{k}$ detect the diagonal (symmetric) part of the pair $(E^1,E^2)$, whereas $\operatorname{ch}_{\BC}$ detects each component separately.
\end{remark}

\subsection{Bicomplex K-theory}\label{subsec:bicomplex_K}

Let $B$ be a connected paracompact Hausdorff space. The \emph{bicomplex $K$-group} $K_{\mathbb{B}\mathbb{C}}(B)$ is defined as the Grothendieck group of the monoid of isomorphism classes of bicomplex vector bundles over $B$, with addition induced by the direct sum $\oplus$ and relations $[E]=[E']+[E'']$ for every short exact sequence
\begin{equation*}
0\longrightarrow E'\longrightarrow E\longrightarrow E''\longrightarrow 0
\end{equation*}
of bicomplex bundles.

The idempotent decomposition is exact, that is such a short exact sequence is equivalent to a pair of short exact sequences of complex bundles $0\to E'^{i}\to E^{i}\to E''^{i}\to 0$ for $i=1,2$. This yields a natural homomorphism
\begin{gather*}
\Phi:K_{\mathbb{B}\mathbb{C}}(B)\to K(B)\times K(B),\\
[E]\mapsto ([E^{1}],[E^{2}]),
\end{gather*}
where $K(B)$ is the complex $K$-group. Because $B$ is connected, the rank of a complex virtual bundle is a single integer, and for every bicomplex bundle we have $\operatorname{rk}(E^{1})=\operatorname{rk}(E^{2})$. Hence the image of $\Phi$ lies in the fibre product
\begin{equation*}
K(B)\times_{\mathbb{Z}}K(B)
=\bigl\{(E^{1},E^{2})\in K(B)\times K(B)\mid 
\operatorname{rk}(E^{1})=\operatorname{rk}(E^{2})\bigr\}.
\end{equation*}

\begin{prop}\label{prop:K_is_fibre_product}
The map $\Phi$ induces an isomorphism
\begin{equation*}
K_{\mathbb{B}\mathbb{C}}(B)\;\cong\;K(B)\times_{\mathbb{Z}}K(B).
\end{equation*}
\end{prop}

\begin{proof}
The homomorphism $\Phi$ is well defined because the idempotent decomposition preserves direct sums and sends short exact sequences to pairs of short exact sequences. To construct an inverse, let $(\eta_{1},\eta_{2})\in K(B)\times_{\mathbb{Z}}K(B)$. Choose representatives $\eta_{1}=[E^{1}]-[F^{1}]$ and $\eta_{2}=[E^{2}]-[F^{2}]$. By adding a sufficiently large trivial bundle to both members of each difference, we may assume that $E^{1},F^{1},E^{2},F^{2}$ are genuine bundles with
\begin{equation*}
\operatorname{rk}(E^{1})-\operatorname{rk}(F^{1})=\operatorname{rk}(E^{2})-\operatorname{rk}(F^{2}).
\end{equation*}
Set
\begin{equation*}
\Psi(\eta_{1},\eta_{2})
=[E^{1}\mathbf{e}\oplus E^{2}\mathbf{e}^{\dagger}]
-\bigl[F^{1}\mathbf{e}\oplus F^{2}\mathbf{e}^{\dagger}]
\in K_{\mathbb{B}\mathbb{C}}(B).
\end{equation*}
This is independent of the chosen representatives and defines a two-sided inverse to $\Phi$.
\end{proof}

The character $\operatorname{ch}_{\BC}$ is nothing but the map
\begin{equation*}
([E^{1}],[E^{2}])\mapsto\operatorname{ch}(E^{1})+\operatorname{ch}(E^{2}),
\end{equation*}
which is a group homomorphism but not a ring homomorphism because the multiplication in $K_{\BC}(B)$ is $(E^{1},E^{2})\otimes(F^{1},F^{2})=(E^{1}\otimes F^{1},E^{2}\otimes F^{2})$, whereas the product in $H^{\mathrm{ev}}(B;\mathbb Q)$ mixes the two factors.  The lift $\widetilde{\operatorname{ch}}_{\BC}$ restores multiplicativity at the expense of enlarging the coefficient ring to $\BC\otimes\mathbb Q$.

\section{(Almost) bicomplex manifolds}

In this final section, we apply the theory of bicomplex characteristic classes to the study of differentiable manifolds equipped with a bicomplex structure on their tangent bundle. This is the bicomplex analogue of the classical notion of an almost complex manifold. We observe that the notion of bicomplex manifolds also appears in \cite{BairdWood} and \cite{SalimovCakan}. 

\begin{defn}\label{def:almost_bicomplex}
Let $M$ be a smooth real manifold. An \emph{almost bicomplex structure} on $M$ is a pair $(I,J)$ of real endomorphisms of the tangent bundle $TM$ such that
\begin{equation*}
I^{2}=J^{2}=-\mathrm{Id},\quad IJ=JI\quad \text{and}\quad I\neq\pm J.
\end{equation*}
The triple $(M,I,J)$ is called an \emph{almost bicomplex manifold}.
\end{defn}

This is the natural infinitesimal version of a bicomplex manifold. Indeed, a bicomplex manifold (defined via holomorphic charts with values in $\BC^n$, see \cite{bra-rab-rom-bicomplex}) induces an almost bicomplex structure on its underlying real tangent bundle, just as a complex manifold induces an almost complex structure.

If $I=J$ (resp. $I=-J$), then $K=IJ=-\mathrm{Id}$ (resp. $K=\mathrm{Id}$), and the $\mathbb{BC}$-action on $TM$ factors through the projection $\mathbb{BC}\to\mathbb{C}$ given by $\mathbf{e}\mapsto 1$, $\mathbf{e}^{\dagger}\mapsto 0$ (resp. $\mathbf{e}\mapsto 0$, $\mathbf{e}^{\dagger}\mapsto 1$). In either case the structure reduces to an ordinary almost complex structure. We exclude these degenerate cases so that the theory captures the genuinely bicomplex phenomenon.

\begin{remark}
    Analogous results to \cite[Appendix~A]{hodge-theory} (to appear in a future article of the second author) allow us to affirm the following: If $(M,I,J)$ is an almost bicomplex manifold, then for each $p \in M$, the pair $(I_p, J_p)$ defines a $\BC$-vector space structure on $T_pM$ by
    \begin{equation*}
        (x + y\,\iunit + z\,\junit + w\,\kunit) \cdot v
        = x v + y I_p(v) + z J_p(v) + w K_p(v),
    \end{equation*}
    where $K_p = I_p J_p$. Since $\BC$ is a real algebra of dimension $4$, the tangent space $T_pM$ must have real dimension divisible by $4$. Hence any almost bicomplex manifold has rank $4n$ for some integer $n$.

    Moreover, because the structure group of the tangent bundle reduces to $\GL(n,\BC)$, which is connected (indeed, $\GL(n,\BC) \cong \GL(n,\C) \times \GL(n,\C)$ and each factor is connected), the manifold $M$ is orientable. In fact, the first bicomplex Chern class $b_1(TM) \in H^2(M;\Z)$ provides an obstruction to the existence of such a structure, as we shall see below.
\end{remark}

\subsection{Examples}

We now give two fundamental examples.

\begin{example}\label{ex:product_complex}
    Let $(M_1, J_1)$ and $(M_2, J_2)$ be almost complex manifolds. Then the product $M = M_1 \times M_2$ admits an almost bicomplex structure given by
    \begin{equation*}
        I = J_1 \oplus J_2 \quad \text{and}\quad J = J_1 \oplus (-J_2),
    \end{equation*}
    where we identify $TM = TM_1 \oplus TM_2$. More explicitly, on $M_1 \times M_2$,
    \begin{equation*}
        I(v_1, v_2) = (J_1 v_1, J_2 v_2), \quad\text{and}\quad J(v_1, v_2) = (J_1 v_1, -J_2 v_2).
    \end{equation*}
    A direct calculation shows that $I^2 = J^2 = -Id$ and $IJ = JI$. Thus every product of two almost complex manifolds carries a natural almost bicomplex structure. This structure is not degenerate if and only if the dimensions of $M_1$ and $M_2$ agree.
\end{example}
If $\sigma(M)$ is the Hirzebruch signature of $M$, $\chi(M)$ is the Euler characteristic class, and $w_2(M)$ is the second Stiefel--Whitney class (see~\cite{Milnor1968}), then
\begin{thm}\label{thm:4d_obstruction}
Let $(M,I,J)$ be a closed almost bicomplex $4$-manifold. Then
\begin{equation*}
3\sigma(M)-2\chi(M)=d^{2}[M]
\end{equation*}
for some class $d\in H^{2}(M;\Z)$ satisfying $d\equiv w_{2}(M)\pmod{2}$.
\end{thm}

\begin{proof}
Since the structure is non-degenerate, $K=IJ$ satisfies $K^{2}=\mathrm{Id}$ and $K\neq\pm\mathrm{Id}$. Hence $K$ has eigenvalues $+1$ and $-1$, and the eigenspaces decompose $TM$ into a sum of complex line bundles $L_{1}\oplus L_{2}$ with respect to the almost complex structure $I$. Write $a_{i}=c_{1}(L_{i})\in H^{2}(M;\Z)$. Then
\begin{equation*}
c_{1}(TM,I)=a_{1}+a_{2}\quad \text{and}\quad c_{2}(TM,I)=a_{1}a_{2}.
\end{equation*}
For any almost complex structure on a closed $4$-manifold, the Hirzebruch signature theorem and the Gauss--Bonnet theorem give the characteristic number identities
\begin{equation*}
c_{1}^{2}[M]=2\chi(M)+3\sigma(M)\quad\text{and}\quad c_{2}[M]=\chi(M),
\end{equation*}
which depend only on the oriented smooth topology of $M$ and not on the chosen almost complex structure. Therefore
\begin{equation*}
(a_{1}-a_{2})^{2}[M]=c_{1}^{2}[M]-4c_{2}[M]=(2\chi+3\sigma)-4\chi=3\sigma-2\chi.
\end{equation*}
Setting $d=a_{1}-a_{2}$, we obtain $d^{2}[M]=3\sigma-2\chi$. Moreover, reducing modulo $2$ and using that $c_{1}(TM,I)\equiv w_{2}(M)\pmod{2}$ for any almost complex structure, we get
\begin{equation*}
d=a_{1}-a_{2}\equiv a_{1}+a_{2}=c_{1}(TM,I)\equiv w_{2}(M)\pmod{2}.\qedhere
\end{equation*}
\end{proof}

\begin{cor}
The complex projective plane $\C\P^2$ admits no almost bicomplex structure.
\end{cor}

\begin{proof}
For $\C\P^2$ the characteristic numbers are $\chi(\C\P^{2})=3$ and $\sigma(\C\P^2)=1$. If $\C\P^2$ carried an almost bicomplex structure, Theorem~\ref{thm:4d_obstruction} would provide a class $d\in H^2(\C\P^2;\Z)\cong\Z a$ with
\begin{equation*}
d^{2}[\C\P^2]=3\sigma-2\chi=-3.
\end{equation*}
Writing $d=ka$ for some $k\in\mathbb{Z}$ and using $a^{2}[\C\P^2]=1$, we obtain $k^2=-3$, which is impossible in $\Z$.
\end{proof}

\begin{remark}\label{rem:chern_detection}
     Chern classes alone are not sufficient in general to detect bicomplex structure. For instance, consider a complex vector bundle that is the sum of four complex line bundles $K$, $L$, $M$, and $N$. We can form two different bicomplex bundles of rank $2$:
    \begin{equation*}
        E = (K \oplus L)\e \oplus (M \oplus N)\ed,
        \qquad
        F = (K \oplus M)\e \oplus (L \oplus N)\ed.
    \end{equation*}
    As complex vector bundles, $E$ and $F$ are isomorphic (both are $K \oplus L \oplus M \oplus N$). However, as bicomplex bundles, they are generally not isomorphic, and their second bicomplex Chern classes differ in general. Thus bicomplex Chern classes carry strictly more information than ordinary complex Chern classes. Nevertheless, just as in the complex case, they are not complete invariants: two non-isomorphic bicomplex structures on the same underlying complex bundle may still have the same bicomplex Chern classes. Further invariants, such as higher characteristic classes or connections, are needed for a full classification.
\end{remark}

\section*{Acknowledgments}

The first-named author gratefully acknowledges the support of the Secretariat of Science, Humanities, Technology and Innovation (SECIHTI), Mexico, which awarded him a sabbatical fellowship.

The second-named author gratefully acknowledges the support of the
Secretaría de Ciencia, Humanidades, Tecnología e Innovación
(SECIHTI), Mexico. This work was partially supported by
DGAPA-PAPIIT project IN101424: Geometría compleja,
singularidades y dinámica.

The third-named author was supported by São Paulo Research Foundation - FAPESP, Brazil, under the grant 25/09846-7.

The fourth-named author was supported by DGAPA PAPIIT IA102626 Singularidades complejas y reales: Invariantes analíticos y topológicos. The second and fourth authors were supported by Simons-UNAM Geometry Program at Cuernavaca Simons Foundation International SFI-MPS-T-Institutes-00011977 JS. 

\bibliographystyle{amsplain}
\bibliography{References}
\end{document}